\documentclass[a4paper, 11pt]{amsart}

\usepackage{amsmath,amssymb,enumitem,verbatim,stmaryrd,xcolor,microtype}
\usepackage[T1]{fontenc}
\usepackage[utf8]{inputenc}
\usepackage[english]{babel}
\usepackage[top=3.5cm,bottom=3.5cm,left=3.2cm,right=3.2cm]{geometry}
\usepackage[linktoc=page,colorlinks,linkcolor={red!80!black},citecolor={red!80!black},urlcolor={blue!80!black}]{hyperref}\newcommand{\arxiv}[1]{\href{http://arxiv.org/pdf/#1}{arXiv:#1}}
\usepackage[all]{xy}
\usepackage[mathcal]{euscript}    
\usepackage{extarrows}                           
\usepackage{mathptmx}                                        
\usepackage{etoolbox}\makeatletter\patchcmd{\@startsection}{\@afterindenttrue}{\@afterindentfalse}{}{}\makeatother    
\patchcmd{\section}{\scshape}{\bfseries}{}{}\makeatletter\renewcommand{\@secnumfont}{\bfseries}\makeatother           
\usepackage[backgroundcolor=yellow,linecolor=yellow,textsize=footnotesize]{todonotes} \makeatletter \providecommand \@dotsep{5} \def\listtodoname{List of Todos} \def\listoftodos{\@starttoc{tdo}\listtodoname} \makeatother 

\pgfrealjobname{eulertype}
\usetikzlibrary{matrix,arrows,decorations.markings}
\usepackage{tikz}\usetikzlibrary{matrix,arrows,decorations.markings}
\tikzset{vertex/.style={rectangle,rounded corners=5pt,draw=black,text height=1.0ex, text depth=0.0ex,font=\footnotesize}}

\DeclareSymbolFont{cmletters}{OML}{cmm}{m}{it}                                     
\DeclareSymbolFont{cmsymbols}{OMS}{cmsy}{m}{n}
\DeclareSymbolFont{cmlargesymbols}{OMX}{cmex}{m}{n}
\DeclareMathSymbol{\myjmath}{\mathord}{cmletters}{"7C}     \let\jmath\myjmath 
\DeclareMathSymbol{\myamalg}{\mathbin}{cmsymbols}{"71}     
\DeclareMathSymbol{\mycoprod}{\mathop}{cmlargesymbols}{"60}\let\coprod\mycoprod
\DeclareMathSymbol{\myalpha}{\mathord}{cmletters}{"0B}     \let\alpha\myalpha 
\DeclareMathSymbol{\mybeta}{\mathord}{cmletters}{"0C}      \let\beta\mybeta
\DeclareMathSymbol{\mygamma}{\mathord}{cmletters}{"0D}     \let\gamma\mygamma
\DeclareMathSymbol{\mydelta}{\mathord}{cmletters}{"0E}     \let\delta\mydelta
\DeclareMathSymbol{\myepsilon}{\mathord}{cmletters}{"0F}   \let\epsilon\myepsilon
\DeclareMathSymbol{\myzeta}{\mathord}{cmletters}{"10}      \let\zeta\myzeta
\DeclareMathSymbol{\myeta}{\mathord}{cmletters}{"11}       \let\eta\myeta
\DeclareMathSymbol{\mytheta}{\mathord}{cmletters}{"12}     \let\theta\mytheta
\DeclareMathSymbol{\myiota}{\mathord}{cmletters}{"13}      \let\iota\myiota
\DeclareMathSymbol{\mykappa}{\mathord}{cmletters}{"14}     \let\kappa\mykappa
\DeclareMathSymbol{\mylambda}{\mathord}{cmletters}{"15}    \let\lambda\mylambda
\DeclareMathSymbol{\mymu}{\mathord}{cmletters}{"16}        \let\mu\mymu
\DeclareMathSymbol{\mynu}{\mathord}{cmletters}{"17}        \let\nu\mynu
\DeclareMathSymbol{\myxi}{\mathord}{cmletters}{"18}        \let\xi\myxi
\DeclareMathSymbol{\mypi}{\mathord}{cmletters}{"19}        \let\pi\mypi
\DeclareMathSymbol{\myrho}{\mathord}{cmletters}{"1A}       \let\rho\myrho
\DeclareMathSymbol{\mysigma}{\mathord}{cmletters}{"1B}     \let\sigma\mysigma
\DeclareMathSymbol{\mytau}{\mathord}{cmletters}{"1C}       \let\tau\mytau
\DeclareMathSymbol{\myupsilon}{\mathord}{cmletters}{"1D}   \let\upsilon\myupsilon
\DeclareMathSymbol{\myphi}{\mathord}{cmletters}{"1E}       \let\phi\myphi
\DeclareMathSymbol{\mychi}{\mathord}{cmletters}{"1F}       \let\chi\mychi
\DeclareMathSymbol{\mypsi}{\mathord}{cmletters}{"20}       \let\psi\mypsi
\DeclareMathSymbol{\myomega}{\mathord}{cmletters}{"21}     \let\omega\myomega
\DeclareMathSymbol{\myvarepsilon}{\mathord}{cmletters}{"22}\let\varepsilon\myvarepsilon
\DeclareMathSymbol{\myvartheta}{\mathord}{cmletters}{"23}  \let\vartheta\myvartheta
\DeclareMathSymbol{\myvarpi}{\mathord}{cmletters}{"24}     \let\varpi\myvarpi
\DeclareMathSymbol{\myvarrho}{\mathord}{cmletters}{"25}    \let\varrho\myvarrho
\DeclareMathSymbol{\myvarsigma}{\mathord}{cmletters}{"26}  \let\varsigma\myvarsigma
\DeclareMathSymbol{\myvarphi}{\mathord}{cmletters}{"27}    \let\varphi\myvarphi
 
\theoremstyle{plain}
\newtheorem{thm}{Theorem}[section]
\newtheorem{cor}[thm]{Corollary}
\newtheorem{lemma}[thm]{Lemma}
\newtheorem{prop}[thm]{Proposition}

\newtheorem*{thmA}{Main theorem}

\theoremstyle{definition}
\newtheorem{df}[thm]{Definition}
\newtheorem{ex}[thm]{Example}
\newtheorem{rem}[thm]{Remark}
\newtheorem*{rem*}{Remark}
\newtheorem*{remA}{Remark 1}
\newtheorem*{remB}{Remark 2}
\newtheorem*{remC}{Remark 3}
\newtheorem*{remD}{Remark 4}

\newcommand{\ses}[3]{0\rightarrow #1\rightarrow #2\rightarrow#3\rightarrow 0}

\newcommand\A{{\mathbb A}}

\newcommand\C{{\mathbb C}}
\newcommand\F{{\mathbb F}}

\newcommand\N{{\mathbb N}}
\renewcommand\P{{\mathbb P}}

\newcommand\Z{{\mathbb Z}}

\newcommand\cB{{\mathcal B}}

\newcommand{\Sc}[2]{\langle #1,#2\rangle}

\newcommand\ud{{\underline{d}}}
\newcommand\ue{{\underline{e}}}
\newcommand\udim{{\underline{\dim}\, }}
\newcommand\Rep{\mathrm{Rep}}
\newcommand\supp{\mathrm{supp}}

\DeclareMathOperator{\Gr}{Gr}
\DeclareMathOperator{\Hom}{Hom}
\DeclareMathOperator{\Ext}{Ext}

\title[Representation type via quiver Grassmannians]{Representation type via Euler characteristics and singularities of quiver Grassmannians}

\author{Oliver Lorscheid}
\address{Instituto Nacional de Matem\'atica Pura e Aplicada, Rio de Janeiro, Brazil}
\email{oliver@impa.br}

\author{Thorsten Weist}
\address{Bergische Universit\"at Wuppertal, Gau\ss str.\ 20, 42097 Wuppertal, Germany}
\email{weist@uni-wuppertal.de}

\begin{document}

\begin{abstract}
 In this text, we characterize the representation type of an acyclic quiver by the properties of its associated quiver Grassmannians. This characterization utilizes and extends known results about singular quiver Grassmannians and cell decompositions into affine spaces.

 While all quiver Grassmannians for indecomposable representations of quivers of finite representation types $A$ and $D$ are smooth and admit cell decompositions, it turns out that all quiver Grassmannians for indecomposable representations of quivers of tame types $A$ and $D$ admit cell decompositions, but some of these quiver Grassmannians are singular (even as varieties). A quiver is wild if and only if there exists a quiver Grassmannian with negative Euler characteristic.
\end{abstract}

\maketitle



\section*{Introduction}
\label{intro}

\subsection*{Motivation}
In this paper, we characterize the representation type of an acyclic quiver in terms of the geometry of the associated quiver Grassmannians. This characterization draws on previous results in the literature, and the proof of this characterization finds its completion in this text. 

Quiver Grassmannians have been studied intensely since their relevance for cluster algebras was revealed. Namely, in case of an acyclic quiver $Q$, the Caldero-Chapoton formula expresses the cluster variables of $Q$ in terms of the Euler characteristics of the associated quiver Grassmannians (see \cite{cc}, \cite{dwz} and \cite{dwz2}).

This discovery started an active search for methods to determine these Euler characteristics and to prove their positivity under suitable assumptions. To highlight some developments, Cerulli Irelli (\cite{Cerulli11}) and subsequently Haupt (\cite{Haupt12}) use torus actions to compute Euler characteristics in terms of torus fixed points. This method leads to satisfactory results for quivers of (extended) Dynkin type $A$. In particular, the Euler characteristics are always nonnegative in type $A$. 

The authors (\cite{LW15a}, \cite{LW15b}) establish cell decompositions into affine spaces for quiver Grassmannians of (extended) Dynkin type $D$. Such a cell decomposition implies that the cohomology is concentrated in even degrees and therefore the Euler characteristic is nonnegative. 

For a while, it was an open problem, which class of projective varieties could be realized as quiver Grassmannians in general. Reineke (\cite{Reineke13}) and Hille (\cite{Hille15}) settle this question: every projective scheme occurs as the quiver Grassmannian of some wild quiver. 

In this paper, we extend the above mentioned results to a classification of the representation type of an acyclic quiver in terms of geometric properties of the associated quiver Grassmannians.

\subsection*{Definition}
Let $Q$ be a quiver, $X$ a finite dimensional complex representation of $Q$ and $\ue$ a dimension vector for $Q$. Then the \emph{quiver Grassmannian $\Gr_\ue(X)$} is defined as the set of $\ue$-dimensional subrepresentations of $Q$. It gains the structure of a projective complex variety by embedding it into the product of the usual Grassmannians $\Gr(\ue_p,X_p)$ over all vertices $p$ of $Q$. 

Let $\ud=\udim X$. By considering $\Gr_\ue(X)$ as the fibre of the universal Grassmannian $\Gr(\ue,\ud)$ over the moduli space of $\ud$-representations of $Q$ with fixed basis, it gains the structure of a scheme. But we will make only implicit references to the schematic structure of the quiver Grassmannian in this text, and the reader might think safely of the quiver Grassmannian as a variety.

\begin{thmA}\label{mainthm}
 Let $Q$ be an acyclic quiver that is not of (extended) Dynkin type $E$. Then the we have the following characterization of the representation type of $Q$.
 \begin{enumerate}
  \item\label{main1} $Q$ is representation finite if and only if all quiver Grassmannians of indecomposable representations of $Q$ are smooth and have a cell decomposition into affine spaces.
  \item\label{main2} $Q$ is tame if and only if all quiver Grassmannians of indecomposable representations of $Q$ have a cell decomposition into affine spaces, but there exist quiver Grassmannians with singularities for indecomposable representations.
  \item\label{main3} $Q$ is wild if and only if every integer can be realized as the Euler characteristic of a quiver Grassmannian for $Q$.
 \end{enumerate}
\end{thmA}

\begin{remA}
 Note that, as explained above, the existence of cell decompositions implies the nonnegativity of the Euler characteristics. It follows that $Q$ is wild if it has a quiver Grassmannian with negative Euler characteristic.
\end{remA}

\begin{remB}
 It is already known for a while that not all quiver Grassmannians for the Kronecker quiver are smooth as schemes. This has been studied in detail in \cite{Cerulli-Esposito11}. For example, for every indecomposable representation $X$ of dimension $(2,2)$, the scheme $\Gr_{(1,1)}(X)$ is a nonreduced point; cf.\ Example 2 in \cite{Cerulli11}. However, this example is regular as a variety. In this paper, we exhibit for every tame quiver $Q$ a quiver Grassmannian that is singular as a variety, including extended Dynkin type $E$.
 
 It also follows from our proof that every wild quiver admits singular quiver Grassmannians (Corollary \ref{cor: singular quiver Grassmannian for wild quivers}).
\end{remB}

\begin{remC} 
 As explained in the proof of the main theorem, the quiver Grassmannians for indecomposable representations of representation finite quivers are smooth, including Dynkin type $E$. Combining this with our result on singular quiver Grassmannians for extended Dynkin type $E$, the assumption that $Q$ is not of type $E$ can thus be removed from the main theorem once we know that every quiver Grassmannian for an indecomposable representation of type $E$ admits a cell decomposition into affine spaces.
 
 At the time of writing, cellular decompositions for type $E$ are investigated in an ongoing collaboration of Giovanni Cerulli Irelli, Francesco Esposito, Hans Franzen and Markus Reineke, as we learned in private communication. There is hope that such decompositions into affine spaces will be established soon.
\end{remC}

\begin{remD}
 During the time of writing, Ringel has proven a result in \cite{Ringel17} that sharpens the last statement of the theorem: every projective scheme is isomorphic to a quiver Grassmannian for any fixed wild acyclic quiver $Q$.  His idea is comparable to the one of Lemma \ref{schofield}.

\end{remD}

\subsection*{Proof of the main theorem}
 It is clear that the characterizations of the different types of quivers are exclusive. In so far, it suffices to establish the respective properties for representation finite, tame and wild quivers.

 Let $Q$ be representation finite. By a result of Caldero and Reineke in \cite{cr}, the quiver Grassmannian $\Gr_\ue(X)$ is smooth if $X$ is an exceptional representation. Since every indecomposable representation of a representation finite quiver $Q$ is exceptional, we conclude that all quiver Grassmannians for indecomposable representations of $Q$ are smooth. As proven in section 3.2 of \cite{LW15a}, every quiver Grassmannian of Dynkin type $A$ or $D$ has a cell decomposition into affine spaces. This shows part \eqref{main1} of the main theorem.
 
 Let $Q$ be tame. By Theorem A in \cite{LW15b}, every quiver Grassmannian for an indecomposable representation of extended Dynkin type $D$ has a cell decomposition into affine spaces. We prove the corresponding result for extended Dynkin type $A$ in this paper. This proof uses different methods for representations in the homogeneous tubes (Theorem \ref{thm: cell deomposition for homogeneous tubes}) and for the other indecomposable representations, which are string modules (Theorem \ref{thm: cell decomposition for string modules}). 
 
 If the Auslander-Reiten quiver of $Q$ has a tube of rank $n\geq 2$, which is the case if $Q$ has at least $3$ vertices, then we exhibit a quiver Grassmannian with Poincar\'e polynomial $2q^2+1$, which cannot come from a smooth projective variety since it fails Poincar\'e duality (Theorem \ref{thm: singular quiver Grassmannians for tubes of large rank}). For the Kronecker quiver, we find a singular quiver Grassmannian in terms of an explicit calculation in coordinates (Theorem \ref{thm: singular quiver Grassmannians for the Kronecker quiver}). This shows part \eqref{main2} of the main theorem.
 
 Let $Q$ be a wild quiver. A theorem of Hille shows that every closed subscheme of $\P^{n-1}$ is isomorphic to a quiver Grassmannian for the $n$-Kronecker quiver. It is an immediate consequence that for $n\geq3$, every integer occurs as the Euler characteristic of a quiver Grassmannian (Corollary \ref{negativeKronecker}). Since every wild quiver contains a minimal wild quiver, it is enough to exhibit quiver Grassmannians with arbitrary Euler characteristics for minimal wild quivers. This reduction leads to a small list of quivers. We show that every quiver Grassmannian of any generalized Kronecker quiver is isomorphic to a quiver Grassmannian of a fixed minimal wild quiver (Proposition \ref{prop: roots with large extension for minimal wild quiver}). As a consequence, every integer occurs as an Euler characteristic of a quiver Grassmannian for a minimal wild quiver (Theorem \ref{thm: negative Euler characteristics for wild quivers}). This shows part \eqref{main3} and finishes the proof of the main theorem.

\subsection*{Complementary results}
 
Beside the main theorem, we prove the following additional facts in this paper.
\begin{itemize}
 \item Every representation infinite quiver has singular quiver Grassmannians (Theorems \ref{thm: singular quiver Grassmannians for tubes of large rank} and \ref{thm: singular quiver Grassmannians for the Kronecker quiver} and Corollary \ref{cor: singular quiver Grassmannian for wild quivers}).
 \item For every tame quiver, there are singular quiver Grassmannians for representations in exceptional and homogeneous tubes (Theorem \ref{thm: singular quiver Grassmannians in homogeneous tubes}).
 \item There are flat families of quiver Grassmannians whose fibres have different isomorphism types, different Poincar\'e polynomials and different Euler characteristics (Example \ref{ex: family of quiver Grassmannians}).
 \item We determine explicit formulae for the $F$-polynomials of all indecomposable representations of the Kronecker quiver (Theorem \ref{thm: F-polynomials for Kronecker}).
\end{itemize}

\subsection*{Acknowledgements} We would like to thank Jan Schr\"oer for sharing his ideas and, in particular, posing the question whether all wild quivers would admit negative Euler characteristics. We would like to thank Alex Massarenti for his help with an example of a singular quiver Grassmannian for the Kronecker quiver. We would like to thank Giovanni Cerulli Irelli and Hans Franzen for their remarks on a first draft of this text.


\section{Cell decomposition for tame quivers}
For an overview concerning the representation theory of (tame) quivers and the well-known results on them, which we use frequently, we refer to \cite[Sections 8,9]{Crawley-Boevey92} and \cite{rin5}.
We fix $k=\C$ as our ground field. We shortly review some basics on quiver representations. Let $Q=(Q_0,Q_1)$ be a quiver with vertex set $Q_0$ and arrow set $Q_1$. We denote arrows of $Q$ by $p\xlongrightarrow{v}q$ or $v:p\to q$ for $p,q\in Q_0$. Throughout the paper, we assume that $Q$ is acyclic, i.e.\ it has no oriented cycles, which means that the corresponding path algebra is finite-dimensional.  
For an arrow $v:p\to q$, let $s(v)=p$ and $t(v)=q$. For a vertex $p\in Q_0$, let 
\[N_p:=\{q\in Q_0\mid\exists \,p\xlongrightarrow{v}q\in Q_1\vee\exists\, q\xlongrightarrow{v}p\in Q_1\}\]
be the set of neighbours of $p$. 

Let $\Rep(Q)$ denote the category of finite-dimensional representations of $Q$. Consider the abelian group
$\mathbb{Z}Q_0=\bigoplus_{q\in Q_0}\mathbb{Z}q$ and its monoid of dimension vectors $\mathbb{N}Q_0$. For a representation $X\in\Rep(Q)$, we denote by $\udim X=\sum_{q\in Q_0}\dim X_q\cdot q$ its dimension vector. On $\Z Q_0$ we have a non-symmetric bilinear form, the Euler form,
which is defined by
\[\Sc{\alpha}{\beta}=\sum_{q\in Q_0}\alpha_q\beta_q-\sum_{v\in Q_1}\alpha_{s(v)}\beta_{t(v)}\]
for $\alpha,\,\beta\in\Z Q_0$. 
Recall that for two representations $X$, $Y$ of $Q$ we have
\begin{align*}\label{HomExt}\Sc{\underline{\dim} X}{\underline{\dim} Y}=\dim_k\Hom(X,Y)-\dim_k\Ext(X,Y)\end{align*}
and $\Ext^i(X,Y)=0$ for $i\geq 2$. For two representations $X$ and $Y$, define $[X,Y]=\dim\Hom(X,Y)$. Finally, we denote by $\tau$ and $\tau^{-1}$ the Auslander-Reiten translation.

If $\delta$ is the unique imaginary Schur root of a tame quiver, the defect of a module $X$ is defined by $\delta(X):=\Sc{\delta}{\udim X}$.

\subsection{Representation theory for $\tilde A_n$}
We first recall some facts on the Auslander-Reiten theory of $\tilde A_n$. Then we briefly explain how covering theory can be used to see that those representations which can be lifted to the universal covering $\widetilde{\tilde A_n}$ are precisely the string modules of $\tilde A_n$. For an introduction to covering theory, we refer to \cite{gab}.

For a fixed orientation of $\tilde A_n$, we can always apply BGP-reflections \cite{bgp} in order to obtain the following orientation
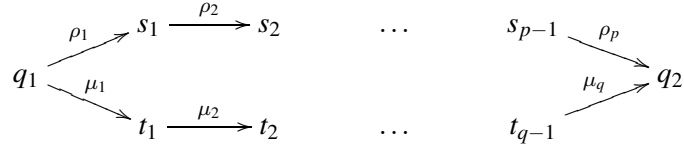
\[
\begin{xy}\xymatrix@R5pt@C30pt{&s_1\ar[r]^{\rho_2}&s_2 &\dots &s_{p-1}\ar[rd]^{\rho_p}\\q_1\ar[ru]^{\rho_1}\ar[rd]^{\mu_1}&&&&&q_2\\&t_1\ar[r]^{\mu_2}&t_2 &\dots &t_{q-1}\ar[ru]^{\mu_q}}\end{xy}
\]
for certain $q, p\geq 1$ with $p+q=n+1$. We denote this quiver by $\tilde A_{p,q}$. The Auslander-Reiten quiver of $\tilde A_{p,q}$ is of the same shape as the one of the original quiver. As we will see, the property of being a string module is preserved under BGP-reflections. This means that we can restrict to this case for the purpose of an overview. 

We first briefly describe the preprojective component of the Auslander-Reiten quiver, the preinjective is obtained dually. The indecomposable projective representations are uniquely determined by their dimension vectors, i.e.
\begin{align*}
 \udim P_{q_1} &= \textstyle \sum_{i=1}^{p-1}s_i+\sum_{i=1}^{q-1}t_i+q_1+2q_2,  & \udim P_{s_j} &= \textstyle \sum_{i=j}^{p-1}s_i+q_2, \\
 \udim P_{q_2} &= q_2,                                                           & \udim P_{t_l} &= \textstyle \sum_{i=l}^{p-1}t_i+q_2
\end{align*}
for $j=1,\ldots,p-1$ and $l=1,\ldots,q-1$. In the case $p=2$ and $q=2$, the preprojective component of the Auslander-Reiten quiver looks as follows. The general case is analogous.
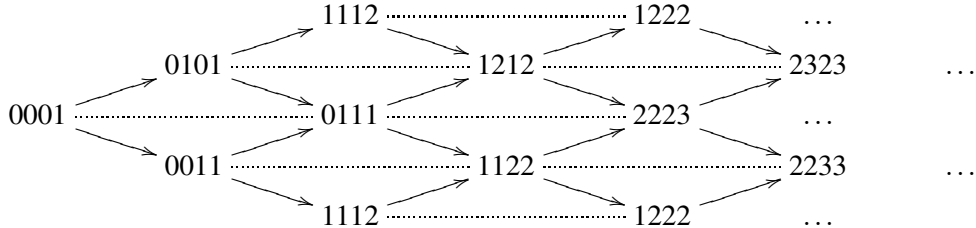
\[
\begin{xy}\xymatrix@R5pt@C30pt{&&1112\ar[rd]\ar@{..}[rr]&&1222\ar[rd]&\dots\\&0101\ar@{..}[rr]\ar[ru]\ar[rd]&&1212\ar@{..}[rr]\ar[ru]\ar[rd]&&2323&\dots\\0001\ar@{..}[rr]\ar[ru]\ar[rd]&&0111\ar@{..}[rr]\ar[ru]\ar[rd]&&2223\ar[ru]\ar[rd]&\dots\\&0011\ar@{..}[rr]\ar[ru]\ar[rd]&&1122\ar@{..}[rr]\ar[ru]\ar[rd]&&2233&\dots\\&&1112\ar@{..}[rr]\ar[ru]&&1222\ar[ru]&\dots}\end{xy}
\]
Here the top and bottom row need to be identified and the order of the dimension vector is given by the ordering $(q_1,s_1,t_1,q_2)$. The dotted lines indicate the Auslander-Reiten translates.

In addition to the preprojective and preinjective component, there is a $\P^1$-family of components which are so-called tubes. All but two of them are of rank one which means that each representation $X$ in such a tube is its own Auslander-Reiten translate $\tau X$. These tubes are called homogeneous. Moreover, there exist two tubes of ranks $p$ and $q$, i.e.\ $\tau^p X=X$ (resp. $\tau^q X=X$)  for every representation $X$ in this tube. We will observe that every representation in one of these tubes is a string module. Tubes which are not of rank one are called exceptional. The quasi-simples in the tube of rank $p$ are given by the simple representations corresponding to the dimension vectors $s_1,\ldots, s_{p-1}$ and to the unique indecomposable representation of dimension $q_1+q_2+\sum_{i=1}^{q-1}t_i$ (with $X_{\rho_1}=0$ if $p=1$).  In turn, the quasi-simples in the tube of rank $q$ are given by the simple representations corresponding to the dimension vectors $t_1,\ldots, t_{q-1}$ and to the unique indecomposable representation of dimension $q_1+q_2+\sum_{i=1}^{p-1}s_i$ (with $X_{\mu_1}=0$ if $q=1$). We denote the corresponding representations by $S_i$ and $T_j$ for $i=1,\ldots,p$ and $j=1,\ldots,q$. Then we have $\tau^{-1}S_i=S_{i+1}$ and $\tau^{-1}S_p=S_1$ and the same is true for the representations $T_i$. It is straightforward to construct all regular representation which are in the same tube recursively. Indeed every representation $R$ in this tube has a quasi-simple subrepresentation $S_i$ such that $R/S_i$ is also regular and in the same tube. Thus all representations are given as middle terms of exact sequences between indecomposable regular representations.

For a quiver $Q$, let $W_Q$ be the free group with generators $\rho\in Q_1$.
We define the universal cover $\tilde Q$ of $Q$ by the vertices $\tilde Q_0= Q_0\times W_Q$ and the arrows $\tilde Q_1=Q_1\times W_Q$ where $(\rho,w):(i,w)\to (j,w\rho)$ for all $\rho:i\to j\in Q_1$ and $w\in W_Q$. Then $\tilde Q$ comes along with a natural map $F:\tilde Q\to Q$ inducing a functor $F:\Rep(\tilde Q)\to \Rep(Q)$, see \cite{gab} for more details. We say that a representation can be lifted to $\tilde Q$ if $F^{-1}(X)$ is not empty. 

\begin{df}\label{string}
We say that a representation $X$ is a string module if it can be lifted to a representation $\tilde X$ of $\tilde Q$ such that $\dim \tilde X_{q,w}\in\{0,1\}$ for all $q\in Q_0$, $w\in W_Q$.
\end{df}
Thus every connected component of the universal covering quiver of $\tilde A_n$ is a quiver of type $A_\infty$. Thus its indecomposable representations are string modules. Note that an indecomposable string module $X$ of $\tilde A_n$ has a unique starting vertex $s_X$ and terminating vertex $t_X$. Moreover, there are two unique vertices $q\in N_{s_X}$ and $q'\in N_{t_X}$ respectively with $\dim X_{q}=\dim X_{q'}=0$. We denote the unique arrow connecting $q$ and $s_X$ by $v(s)$ and unique arrow connecting $q'$ and $t_X$ by $v(t)$.

\begin{lemma}\label{orientation}Let $X$ be an indecomposable string module of $\tilde A_n$. Then $X$ is preprojective if and only if $v(s_X)$ and $v(t_X)$ are oriented towards $s_X$ and $t_X$, preinjective if and only if $v(s_X)$ and $v(t_X)$ are oriented away from $s_X$ and $t_X$ and regular otherwise.
\end{lemma}
\begin{proof}This is clearly true for simple representations. As every preprojective (resp. preinjective) representation can be obtained from a simple projective representation (possibly of another quiver) by a series of BGP-reflections at sources (resp. sinks) which become sinks (resp. sources) after reflecting, the claim follows by induction. Note that we can apply BGP-reflections on the universal covering.
\end{proof}

\begin{lemma}
 Every indecomposable representation of $\tilde A_n$ that lies in the preinjective or preprojective component or in an exceptional tube of the Auslander-Reiten quiver is a string module.
\end{lemma}

\begin{proof} We will use well-known facts on tree modules throughout the proof. For more details on tree modules we refer to \cite{wei}. It is clear that the projective and injective representations can be lifted to $\widetilde{\tilde A_n}$. Moreover, it is straightforward to check that the lifting property is provided under Auslander-Reiten translation, i.e. $\tau_{\tilde Q}^{-1} \tilde X =\widetilde{\tau_Q^{-1}X}$ for non-injective representations and  $\tau_{\tilde Q} \tilde X =\widetilde{\tau_QX}$ for non-projective representations. Thus the claim for preprojective and preinjective representations follows.

Fix an exceptional tube of rank $m$. Then it contains $m$ indecomposables of dimension $n\delta$ for each $n\geq 1$. It also contains $m(m-1)$ exceptional representations of dimension $\alpha<\delta$. These exceptional representations are tree modules by \cite{rin1} and thus string modules. Indeed, it is well-known that all tree modules can be lifted to the universal covering.
An arbitrary representation $X$ in this tube is obtained recursively as middle term of an exact sequence of the form $\ses{X_1}{X}{X_2}$ where $X_1$ and $X_2$ are representations lying in the same tube satisfying $\Ext(X_2,X_1)=k$. Thus a basis element of $\Ext(X_2,X_1)=k$ can be chosen in such a way that it corresponds to an arrow of $\tilde A_n$. This shows that $X$ is also a tree module and thus a string module if $X_1$ and $X_2$ are. 
\end{proof}

\begin{rem} 
Let $\delta=(1,\ldots,1)$ be the unique imaginary Schur root of $\tilde A_n$. If $X$ is a fixed preprojective representation, then $\udim X+\delta$ is also a preprojective root. Moreover, the string corresponding to $\udim X+\delta$ is obtained by glueing the appropriate string module of dimension $\delta$ to it. It can be checked that all preprojective representations are obtained in this way.

Analogously, if $\alpha$ is a regular root of $\tilde A_n$, then $\alpha+\delta$ is also a regular root. The corresponding indecomposable of dimension $\alpha+\delta$ is obtained in the same manner.
\end{rem}

\subsection{Homogeneous tubes}
As part of the results about quiver Grassmannians of extended Dynkin type $D$, the authors show in section 1.7 of \cite{LW15b} that all quiver Grassmannians for a indecomposable representation in a homogeneous tube admit a cell decomposition into affine spaces. However, the proof of this result does not rely on any particular properties of type $D$, but applies to all tame quivers, including extended Dynkin type $E$. Therefore, we have:

\begin{thm}\label{thm: cell deomposition for homogeneous tubes}
 Let $Q$ be a tame quiver and $X$ an indecomposable representation in a homogeneous tube. Then every quiver Grassmannian for $X$ admits a cell decomposition into affine spaces.
\end{thm}


\subsection{String modules and extended Dynkin type A}

Let $Q$ be a quiver of extended Dynkin type $\widetilde A_{n-1}$. Then all indecomposable representations of $Q$, but those in the homogeneous tubes, are string modules. For these particular string modules, we can apply the techniques of \cite{L13} and \cite{LW15a} to establish cell decompositions into affine spaces. 

All indecomposable string modules $X$ of $Q$ have a basis $\cB$ such that the coefficient quiver $\Gamma=\Gamma(X,\cB)$ is as depicted in the following illustration. The canonical map $\pi:\Gamma\to Q$ corresponds to the vertical projection in this picture.
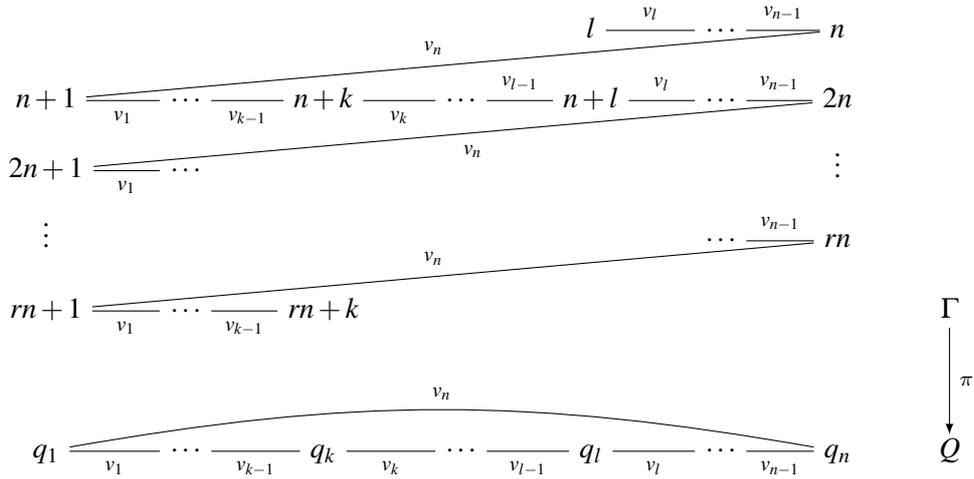
\[
 \beginpgfgraphicnamed{tikz/fig1}
 \begin{tikzpicture}[>=latex]
  \matrix (m) [matrix of math nodes, row sep=1.3em, column sep=2.3em, text height=1ex, text depth=0ex]
   {          &       &       &       &   l    &\dotsb &   n   \\   
         n+1  &\dotsb & n+k   &\dotsb &  n+l   &\dotsb &  2n   \\   
         2n+1 &\dotsb &       &       &        &       &\vdots \\   
       \vdots &       &       &       &        &\dotsb &  rn   \\   
        rn+1  &\dotsb & rn+k  &       &        &       &      & \Gamma \\   
          \   &       &       &       &        &       &                \\   
         q_1  &\dotsb &  q_k  &\dotsb &  q_l   &\dotsb &  q_n & Q      \\   
};
   \path[-,font=\scriptsize]
   (m-1-5) edge node[auto] {$v_l$} (m-1-6)
   (m-1-6) edge node[auto] {$v_{n-1}$} (m-1-7)
   (m-2-1) edge node[auto] {$v_n$} (m-1-7)
   (m-2-1) edge node[auto, swap] {$v_1$} (m-2-2)
   (m-2-2) edge node[auto, swap] {$v_{k-1}$} (m-2-3)
   (m-2-3) edge node[auto, swap] {$v_k$} (m-2-4)
   (m-2-4) edge node[auto] {$v_{l-1}$} (m-2-5)
   (m-2-5) edge node[auto] {$v_l$} (m-2-6)
   (m-2-6) edge node[auto] {$v_{n-1}$} (m-2-7)
   (m-3-1) edge node[auto,swap] {$v_n$} (m-2-7)
   (m-3-1) edge node[auto, swap] {$v_1$} (m-3-2)
   (m-4-6) edge node[auto] {$v_{n-1}$} (m-4-7)
   (m-5-1) edge node[auto] {$v_n$} (m-4-7)
   (m-5-1) edge node[auto, swap] {$v_1$} (m-5-2)
   (m-5-2) edge node[auto, swap] {$v_{k-1}$} (m-5-3)
   (m-7-1) edge node[auto, swap] {$v_1$} (m-7-2)
   (m-7-2) edge node[auto, swap] {$v_{k-1}$} (m-7-3)
   (m-7-3) edge node[auto, swap] {$v_{k}$} (m-7-4)
   (m-7-4) edge node[auto, swap] {$v_{l-1}$} (m-7-5)
   (m-7-5) edge node[auto, swap] {$v_l$} (m-7-6)
   (m-7-6) edge node[auto, swap] {$v_{n-1}$} (m-7-7)
   (m-7-1) edge[bend left=10] node[auto] {$v_n$} (m-7-7)
   ;
   \path[->,font=\scriptsize]
   (m-5-8) edge node[auto] {$\pi$} (m-7-8)
   ;
  \end{tikzpicture}
\endpgfgraphicnamed
\]
Note that the arrows of $Q$ can be arbitrarily oriented and that we allow the case that $l\leq k$, which means that the vertices $q_k$ and $q_l$ have to change positions in the above picture. 

Let $\ue$ be a dimension vector for $Q$. A subset $\beta$ of $\Gamma_0$ is of type $\ue$ if $\beta\cap\pi^{-1}(p)$ has cardinality $\ue_p$ for every $p\in Q$. A subset $\beta$ of $\Gamma_0$ is \emph{successor closed} if for every arrow $v:s\to t$ in $\Gamma$ with $s\in\beta$, we also have $t\in\beta$.

\begin{thm}\label{thm: cell decomposition for string modules}
 Let $X$ be an irreducible string module and $\ue$ a dimension vector of $Q$. Then $\Gr_\ue(X)$ has a cell decomposition into affine spaces. The cells $C_\beta^X$ of this decomposition are labelled by the successor closed subsets $\beta$ of type $\ue$. Consequently, the Euler characteristic of $\Gr_\ue(X)$ equals the number of successor closed subsets of $\Gamma_0$.
\end{thm}

\begin{proof}
 Note that if $\Gr_\ue(X)$ has a cell decomposition into affine spaces, then its cohomology is concentrated in even degrees and its Euler characteristic equals the number of cells. Therefore the last claim of the theorem follows once the cell decomposition and the labelling of the cells is established.

 The existence of a cell decomposition into affine spaces follows easily from the results in either \cite{L13} or \cite{LW15a}. Both proofs are based on certain tools and properties---ordered polarizations and relevant maximal pairs in the former case and Schubert systems in the latter case. Since the introduction of these notions would require more space than the actual proof, we choose to don't burden this paper with lengthy expositions, but restrict ourselves to the outline of both proofs and refer the reader to the corresponding paper for definitions. In particular, we like to mention that the general case is proven analogously to the special case where $Q$ is the Kronecker quiver and $X$ is a preprojective representation, cf.\ Example 4.5 in \cite{L14} for the former method and Proposition 3.1 in \cite{LW15a} for the latter method.
 
 As a first common step, we note that the preinjective representations $X$ of $Q$ stay in natural correspondence to the preprojective representations $X^\ast$ of $Q^\ast$ where $Q^\ast$ results from $Q$ by reversing all arrows. This association defines an isomorphism $\Gr_\ue(X)\to\Gr_{\ue^\ast}(X^\ast)$ of quiver Grassmannians where $\ue^\ast=\udim X-\ue$. See section 1.8 in \cite{LW15b} for details.
 
 This correspondence reduces the proof to preprojective representations and representations in an exceptional tube. Let $\cB$ be the ordered basis as depicted in the illustration above. Note that for preprojective $X$, the arrow $v_k$ is oriented towards $q_k$, see Lemma \ref{orientation}. If $X$ is in an exceptional tube, then we can also assume that $v_k$ is oriented towards $q_k$. If this was not the case, we can use the reverse order of $\cB$, i.e.\ exchange $i\in\cB$ by $rn+k+l-i$, and relabel the vertices of $Q$ correspondingly to exchange the roles of $q_k$ and $q_l$, so that our assumption is satisfied.
 
 \textit{First proof:} Theorem 4.1 of \cite{L13} provides a cell decomposition of $\Gr_\ue(X)$ into affine spaces provided that $X$ admits an ordered polarization (cf.\ \cite[section 3.3]{L13}) such that every relevant pair (cf.\ \cite[section 2.3]{L13}) is maximal for at most one arrow of $Q$ (cf.\ \cite[section 3.4]{L13}). The same theorem states that the cells $C_\beta^X$ are labelled by the extremely successor closed subsets $\beta$ of $\Gamma_0$ (cf.\ \cite[section 3.1]{L13}). Since $\pi:\Gamma\to Q$ is unramified (cf.\ \cite[section 3.2]{L13}), a subset $\beta$ of $\Gamma_0$ is extremal successor closed if and only if it is successor closed (cf.\ \cite[section 3.1]{L13}).
 
 We indicate why these hypotheses are satisfied for the chosen ordered basis $\cB$. Thanks to the simple shape of the coefficient quiver, it can be seen immediately that $\cB$ is a polarization. That $\cB$ is an ordered polarization follows from the fact that in the above illustration of $\Gamma$ we do not have arrows crossing each other. That every relevant pair $(i,j)\in\cB\times\cB$ is maximal for at most one arrow of $Q$ follows from the shape of $\Gamma$ and the specific ordering of $\cB$.
 
 \textit{Second proof:} The reduced Schubert system $\overline\Sigma=\overline\Sigma(X,\cB)$ (cf.\ \cite[Def.\ 2.12]{LW15a}) admits a patchwork $\{\Xi_j\}_{j=1,\dotsc,s}$ (cf.\ \cite[Def.\ 2.31]{LW15a}) with $s=r-1$ if $k<l$ and $s=r$ if $k\geq l$ whose patches $\Xi_j$ are as follows:
 \[
 \beginpgfgraphicnamed{tikz/fig2}
  \begin{tikzpicture}[>=latex]
   \matrix (m) [matrix of math nodes, row sep=1em, column sep=2em, text height=1.5ex, text depth=0.5ex]
    { \node[vertex](b-1-k){v_{l-1},l,jn+l-1}; & \node[vertex](1-k+1){l,jn+l}; & \node[vertex](a-1-k+2){v_l,l',jn+l''}; & \dotsb & \node[vertex](2n+1-k-2n+1){(r-j)n+k,rn+k}; 
\\ };
    \path[-,font=\scriptsize]
    (b-1-k) edge (1-k+1)
    (a-1-k+2) edge (1-k+1)
    (a-1-k+2) edge (m-1-4)
    (m-1-4) edge (2n+1-k-2n+1)
    ;
    \path[dotted,font=\scriptsize]
    ;
  \end{tikzpicture}
 \endpgfgraphicnamed
 \]
 where the relevant triple $(v_{l-1},l,jn+l-1)$ (cf.\ \cite[section 2.1]{LW15a}) appears as a vertex if and only if the arrow $v_{l-1}$ is oriented towards $q_l$, which is the case for preprojective $X$, and where $l'$ and $l''$ are $l$ or $l+1$, depending on the orientation of $v_l$.
 
 Each patch $\Xi_j$ is an extremal path (cf.\ \cite[Def.\ 2.35]{LW15a}). By Corollary 2.37 in \cite{LW15a}, each extremal path has an extremal solution (cf.\ \cite[Def.\ 2.27]{LW15a}), and therefore by Corollary 2.34 in \cite{LW15a}, the reduced Schubert system $\overline\Sigma$ is totally solvable (cf.\ \cite[section 2.8]{LW15a}). By Corollary 2.20 in \cite{LW15a}, the quiver Grassmannian $\Gr_\ue(X)$ has a cell decomposition into affine spaces whose cells are labelled by the non-contradictory subsets $\beta$ of $\Gamma_0$ (cf.\ \cite[section 2.3]{LW15a}). Since $\pi:\Gamma\to Q$ is unramified, $\beta$ is non-contradictory if and only it is successor closed, thus the theorem.
\end{proof}

\begin{rem}
 Note that the characterization of the Euler characteristic in terms of successor closed subsets is not new. Haupt proves this result for any unramified tree module in \cite{Haupt12}, using an idea of Cerulli Irelli from \cite{Cerulli11}. 
 \end{rem}                                                
\subsection{$F$-polynomials}
In this section, we calculate the generating function of Euler characteristics of quiver Grassmannians for some representations of extended Dynkin quivers. The methods are analogous to those of \cite[Section 1.7, Section 4]{LW15b}. Recall that for a representation $X$, its $F$-polynomial $F_X\in\C[x_q\mid q\in Q_0]$ is defined by
\[F_X:=\sum_{\ue\in\N Q_0}\chi(\Gr_{\ue}(X))x^{\ue}\]
where $x^{\ue}:=\prod_{q\in Q_0}x_q^{e_q}$.

First we investigate $F$-polynomials of representations from homogeneous tubes. Thus let $X_{n\delta}$ be any indecomposable representation of an extended Dynkin quiver which lies in a homogeneous tube and which is of dimension $n\delta$. Moreover, we denote by $F_{n\delta}$ its $F$-polynomial. Note that this notation is not misleading because we have $F_{X_{n\delta}}=F_{X'_{n\delta}}$ for two representations of dimension $n\delta$ from two different homogeneous tubes. Moreover, define
\[z=\frac{1}{2}\sqrt{F_{\delta}^2-4x^{\delta}},\quad\lambda_{\pm}=\frac{F_{\delta}}{2}\pm z.\]

As a consequence of Theorem \ref{thm: cell deomposition for homogeneous tubes}, we obtain the following result, see \cite[Corollary 1.23, Corollary 4.12]{LW15b}:
\begin{thm}\label{Fpoly}
Let $F_{X_{-1}}=F_{X_0}=1$. For $n\geq 1$ we have
\[F_{n\delta}=F_\delta F_{(n-1)\delta}-x^\delta F_{(n-2)\delta}=\frac{1}{2z}(\lambda_+^{n+1}-\lambda_-^{n+1}).\]
\end{thm}

We also want to describe how to obtain the $F$-polynomials for the representations of $K(2)$ in a rather straightforward way. We get results which are comparable to those obtained in \cite[Section 4]{LW15b}. Note that the case $\tilde A_n$ is a bit more tedious than the case of the Kronecker quiver. But it is also treatable with the methods we present here or in \cite[Section 4]{LW15b}. 

Let $P_0$ and $P_1$ with $\udim P_0=(1,2)$ and $\udim P_1=(0,1)$ be the indecomposable projective representations of $K(2)$ where we denote the vertices by $0$ and $1$ and the arrows by $a$ and $b$. Then every preprojective representation is an Auslander-Reiten translate of either $P_0$ or $P_1$ and thus of dimension $(n,n+1)$ for some $n\geq 2$. We denote it by $X_n$. It has a coefficient quiver of the form
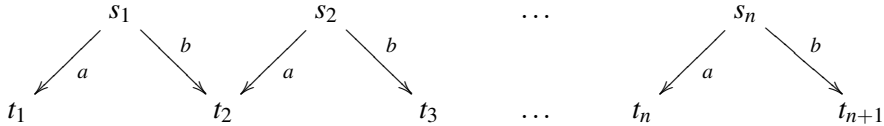
\[\xymatrix{&s_1\ar[ld]^a\ar[rd]^b&&s_2\ar[ld]^a\ar[rd]^b&&\ldots&&s_n\ar[ld]^a\ar[rd]^b&\\t_1&&t_2&&t_3&\ldots&t_n&&t_{n+1}&}\] 
 with $n$ sources and $n+1$ sinks. We denote the corresponding basis by $\cB_n$.

In order to determine the Euler characteristic $\chi(\Gr_{(c,d)}(X_n))$, we have to count the number of successor closed subsets of $\cB_n$ of type $(c,d)$, i.e. with $c$ sources and $d$ sinks. Let $x=x_0$ and $y=x_1$. Then we obtain the following recursive formula:
\begin{lemma}For the $F$-polynomials of preprojective representations of $K(2)$ we have
\[F_{X_n}=(1+y+xy)F_{X_{n-1}}-xyF_{X_{n-2}}=F_\delta F_{X_{n-1}}-x^\delta F_{X_{n-2}}\]for $n\geq 1$ and where $F_{X_{-1}}:=1$ and $F_{X_0}=1+y$.
\end{lemma}
\begin{proof}
Every successor closed subset of $\cB_n$ yields a pair of successor closed subsets of $\cB_{n-1}$ and the basis $\{s_n,t_{n+1}\}$ of the representation $T_b$ with coefficient quiver $s_n\xrightarrow{b} t_{n+1}$. Note that the coefficient quiver of $X_n$ is obtained by glueing these two coefficient quivers by the arrow $a$. Moreover, we have $F_{T_b}=1+y+xy$. The other way around a pair $(S,T)$ of successor closed subsets of $\cB_{n-1}$ and $\{s_n,t_{n+1}\}$ does {\it not} give rise to a successor closed subset of $\cB_n$ if and only if $T=\{s_n,t_{n+1}\}$ and $S$ does not contain $t_n$. But this already means that it does not contain $s_{n-1}$. In turn $S$ is already a successor closed subset of $\cB_{n-2}$.
\end{proof}

As we have $xy=x^\delta$ and $F_\delta=1+y+xy$, with $z$ and $\lambda_{\pm}$ as above, we obtain
\begin{eqnarray*}\begin{pmatrix}F_{X_{n}}\\F_{X_{n+1}}\end{pmatrix}&=&\begin{pmatrix}0&1\\-x^{\delta}&F_{\delta}\end{pmatrix}^{n+1}\begin{pmatrix}F_{X_{-1}}\\F_{X_0}\end{pmatrix}\\&=&\frac{1}{-2z}\begin{pmatrix}-1&-1\\-\lambda_+&-\lambda_-\end{pmatrix}\begin{pmatrix}\lambda_+&0\\0&\lambda_-\end{pmatrix}^{n+1}\begin{pmatrix}-\lambda_-&1\\\lambda_+&-1\end{pmatrix}\begin{pmatrix}F_{X_{-1}}\\F_{X_0}\end{pmatrix}
\end{eqnarray*}
Thus we get
\[F_{X_n}=\frac{1}{2z}\begin{pmatrix}\lambda_-^{n+1}\lambda_+-\lambda_+^{n+1}\lambda_-, &\lambda^{n+1}_+-\lambda^{n+1}_-\end{pmatrix}\begin{pmatrix}F_{X_{-1}}\\F_{X_0}\end{pmatrix}.\]
Applying Theorem \ref{Fpoly} and, moreover, $\lambda_+\lambda_-=x^\delta$, we obtain the following result:
\begin{thm}\label{thm: F-polynomials for Kronecker}
 For the $F$-polynomial of the preprojective representations of $K(2)$, we have \[F_{X_n}=F_{n\delta}F_{X_0}-x^\delta F_{(n-1)\delta}.\]
\end{thm}
Thus the $F$-polynomial depends only on the $F$-polynomials of the homogeneous tubes and of the simple projective representation. This phenomenon can also be found in the case of extended Dynkin quivers of type $\tilde D_n$. It is likely that one obtains similar formulae in the general case $\tilde A_n$.


\section{Singular quiver Grassmannians for tame quivers}

In this section, we prove that every tame quiver $Q$ admits a quiver Grassmannian with singularities. 

\subsection{Tame quivers with at least three vertices}
With exception of the Kronecker quiver, every tame acyclic quiver has a tube of rank $n\geq2$. We utilize this fact to exhibit singular quiver Grassmannians of a small dimensional representation in such an exceptional tube. By $X_{S,n\delta}$, we denote the unique indecomposable representation of dimension $n\delta$ in an exceptional tube $\mathcal T$ which has the quasi-simple representation $S$ as a subrepresentation.

\begin{lemma}\label{lem1}
 Let $\mathcal T$ be an exceptional tube of rank two. Then there exists a quasi-simple representation $S$ in $\mathcal T$ and a projective subrepresentation $P$ of $\tau^{-1} S$ of defect $\Sc{\delta}{\udim P}=-1$ such that $P\in{}^\perp S$ and $\Ext(S,P)\cong\Ext(X_{S,\delta},P)$. 
\end{lemma}

\begin{proof}
 Let $T=\tau S=\tau^{-1} S$. Then we have $\Hom(S,T)=0=\Hom(T,S)$ and there exists an Auslander-Reiten sequence $\ses{S}{X_{S,\delta}}{T}$. If $P$ is a subrepresentation of $X_{S,\delta}$, we have a commutative diagram
 \[\xymatrix{0\ar[r] &S\ar[r] & X_{S,\delta}\ar[r] & T\ar[r] & 0 \\
  0\ar[r]& P'\ar[r]\ar[u] & P\ar[u]\ar[r]& P''\ar[r]\ar[u]& 0}\] 
If $P''$ is a proper subrepresentation of $T$ - because $T$ is regular and quasi-simple - it cannot be preinjective or regular which means that it has negative defect. The same holds for $P'$. Let $P$ be a projective subrepresentation of $X_{S,\delta}$ of minimal dimension among those projective subrepresentations satisfying the condition $\delta(P)=-1$ (which exists for every tame quiver). This yields $\delta(P'')=-1$ or $\delta(P')=-1$ because the defect is additive on exact sequences and $\delta(S)=\delta(T)=0$. By minimality, $P''=T$ is not possible. Also the case $P'=S$ is not possible because the embedding $P''\to T$ factors through $X_{S,\delta}$ because $\Ext(P'',S)=0$. This already shows that $P''=0$ or $P'=0$. In turn, either $S$ or $T$ has a projective subrepresentation of defect $-1$.

 Thus we might assume that $P$ is an indecomposable projective subrepresentation of $T$ of defect $-1$ (otherwise we may consider the exact sequence $\ses{T}{X_{T,\delta}}{S}$ together with the same projective representation $P$). Then the cokernel $I:=T/P$ has defect $1$ and is preinjective because $T$ is quasi-simple. In particular, it is indecomposable. Indeed, every summand of $I$ must have positive defect. Since $I$ is preinjective, we have $\Hom(I,T)=0$ and thus $\dim\Ext(I,S)\leq\dim\Ext(I,X_{S,\delta})=1$. Considering the long exact sequence
 \[0\to \Hom(I,S)\to\Hom(T,S)\to\Hom(P,S)\to\Ext(I,S)\to\Ext(T,S)\to\Ext(P,S)=0\]
 we obtain $\Ext(I,S)=\Ext(T,S)=\C$ and thus $0=\Hom(T,S)=\Hom(P,S)$. This means $P\in{}^\perp S$.

 Since $\mathcal T$ is of rank two, we have $\Ext(S,T)=\C$. Thus it follows that $\dim\Ext(S,P)\geq 1$ because $\Ext(S,I)=0$. As $P$ is of defect $-1$, we have $\Ext(X_{S,\delta},P)=\C$ which yields $\Ext(X_{S,\delta},P)\cong\Ext(S,P)$.
\end{proof}

\begin{prop}\label{notsmooth}
 Let $\mathcal T$ be a tube of rank two and assume that $S,T=\tau^{-1}S, P$ and $X_{S,\delta}$ are as constructed in Lemma \ref{lem1}. Moreover, consider the short exact sequence sequence $\ses{X_{S,\delta}}{X_{S,2\delta}}{X_{S,\delta}}$. Then the quiver Grassmannian $\Gr_{\udim P+\udim S}(X_{S,2\delta})$ is not smooth as a variety.
\end{prop}

\begin{proof}
 As $\Ext(S,P)=\C$ and $S\in{}^\perp P$, there exists a short exact sequence $\ses{P}{P'}{S}$ with indecomposable middle term. Since $\delta(P')=-1$, the representation $P'$ is preprojective. Consider the map $\Psi_{\ue}:\Gr_\ue(X_{2\delta})\to\coprod_{\underline f+\underline g=\ue}\Gr_{\underline f}(X_{S,\delta})\times\Gr_{\underline g}(X_{S,\delta})$ where $\ue:=\udim S+\udim P$. Every $U\subseteq X_{S,2\delta}$ induces a commutative diagram
 \[\xymatrix{0\ar[r] &X_{S,\delta}\ar[r] & X_{S,2\delta}\ar[r] & X_{S,\delta}\ar[r] & 0 \\
  0\ar[r]& A\ar[r]\ar[u] & U\ar[u]\ar[r]& V\ar[r]\ar[u]& 0}\]

 Let $U\cong U_1\oplus\ldots U_r$ be the direct sum decomposition of $U$. Then each $U_i$ is either preprojective or regular, i.e. $\Sc{\udim U_i}{\delta}\geq 0$. Since $1=\Sc{\udim U}{\delta}=\sum_{i=1}^r\Sc{\udim U_i}{\delta}$, there must be precisely one preprojective summand. We have $\udim U=\udim S+\udim P\leq \udim S+\udim \tau^{-1} S=\delta$, $\Hom(S,X_{S,\delta})=\C$ and, moreover, $S$ is the only quasi-simple subrepresentation of $X_{S,\delta}$. Thus there can be at most one regular direct summand which is forced to be $S$. This yields that $U\cong P\oplus S$ or $U\cong P'$.

 The same holds for $A$ and $V$, i.e.\ they can either be preprojective or regular or a direct sum of both. As the defect is additive, only one of the two representations can have a preprojective direct summand. If one summand is regular, it is forced to be isomorphic to $S$ because it is the only regular subrepresentation of $X_{S,\delta}$. As $U\cong P\oplus S$ or $U\cong P'$, the representations $A$ and $V$ can at most have one regular direct summand in total. Indeed, neither $P$ nor $P'$ have a subrepresentation which is isomorphic to $S$ and, moreover, $\Hom(P,S)=0$ and $\Hom(P',S)=\C$. Thus we obtain $(A,V)\in\mathcal X=\{(0,P'),\,(P',0),\,(P,S),\,(S,P),\,(P\oplus S,0),\,(0,P\oplus S)\}$.

 Clearly, we have $\Gr_{\udim P}(X_{S,\delta})\cong\Gr_{\udim S}(X_{S,\delta})=\{\mathrm{pt}\}$ as $S$ is quasi-simple and $P$ projective of defect $-1$. We have $\Ext(P',X_{S,\delta})=0$ and thus $\Hom(P',X_{S,\delta})=\C$.
 But $P'$ is not a subrepresentation of $X_{S,\delta}$ as the only homomorphism (up to scalars) from $P'$ to $X_{S,\delta}$ factors through $S$. 
But $P\oplus S$ is a subrepresentation of $X_{S,\delta}$ in a unique way. 
Indeed, as $P$ is a projective subrepresentation of $T$ with $\Hom(P,T)=\C$  and $S\in P^\perp$, the unique embedding of $P$ into $T$ factors through $X_{S,\delta}$. Thus we obtain $\Gr_{\ue}(X_{S,\delta})=\{\mathrm{pt}\}$. This means that we have
 \[\Gr_{\ue}(X_{S,2\delta})=\bigsqcup_{(A,V)\in\mathcal X}\Psi_\ue^{-1}(A,V). \]
Let us investigate the fibres using that \cite[Lemma 3.11]{cc} generalizes to arbitrary exact sequences. This means that $\Psi_\ue^{-1}(A,V)=\A^{[V,X_{S,\delta}/A]}$ if it is not empty.

If $(A,V)=(0,P\oplus S)$, the fibre is empty because $\Hom(S,X_{S,2\delta})=\C$ and the only homomorphism factors through the first copy of $X_{S,\delta}$.

 If $(A,V)=(P\oplus S, 0)$, the fibre is clearly not empty and thus a point.

 If $(A,V)=(P,S)$, applying $\Hom(S,\underline{\quad})$ and $\Hom(\underline{\quad},X_{S,\delta})$ we get isomorphisms $\Ext(S,P)\cong \Ext(S,X_{S,\delta})$ (by construction) and $\Ext(X_{S,\delta},X_{S,\delta})\cong\Ext(S,X_{S,\delta})$. This means there exists a commutative diagram
 \[\xymatrix{0\ar[r] &X_{S,\delta}\ar[r] & X_{S,2\delta}\ar[r] & X_{S,\delta}\ar[r] & 0 \\
  0\ar[r]& P\ar[r]\ar[u] & P'\ar[u]\ar[r]& S\ar[r]\ar[u]& 0}\]
with injective vertical maps. Thus the fibre is not empty.
Now $\Ext(S,P)\cong \Ext(S,X_{S,\delta})$ together with $\Hom(S,P)=0$ implies $\C=\Hom(S,X_{S,\delta})\cong\Hom(S,X_{S,\delta}/P)$. Thus we get $\mathbb A^{[S,X_{S,\delta}/P]}= \mathbb A^1$.

 If $(A,V)=(S,P)$, the fibre is not empty because the inclusion $P\hookrightarrow X_{S,\delta}$ factors through $X_{S,2\delta}$ as $P$ is projective. Thus the fibre is $\mathbb A^{[P,T]}=\mathbb A^1$ because $\Hom(P,T)\cong\Hom(P,X_{S,\delta})$ which follows from $P\in {}^\perp S$. 

This shows that $\Gr_\ue(X_{S,2\delta})$ has a cell decomposition into affine spaces consisting of one point and two affine lines. In particular, we obtain that $P_{\Gr_{\ue}(X_{S,2\delta})}=2q^2+1$. Since the constant term $1$ is the dimension of the zeroth singular homology $H_0(\Gr_{\ue}(X_{S,2\delta});\C)$, which counts the number of connected components, the variety $\Gr_{\ue}(X_{S,2\delta})$ is connected. If it was nonsingular as a variety, then it would satisfy Poincaré duality, which is not the case since the coefficients of the Poincar\'e polynomial are not symmetric.
\end{proof}

\begin{thm}\label{thm: singular quiver Grassmannians for tubes of large rank}
 For every extended Dynkin quiver $Q$ with $|Q_0|\geq 3$, there exists a singular quiver Grassmannian $\Gr_e(X)$ where $X$ is an indecomposable representation lying in an exceptional tube.
\end{thm}

\begin{proof}
 If $Q$ is not of type $\tilde A_n$, it has an exceptional tube of rank two, see \cite{Dlab-Ringel76} or \cite[\S 9]{Crawley-Boevey92}. Thus we can combine Lemma \ref{lem1} and Proposition \ref{notsmooth}. 

 If $Q=\tilde A_n$ (non-cyclic) with $|Q_0|\geq 3$, there exists a subquiver $q_1\xrightarrow{\rho_1}q\xleftarrow{\rho_2}q_2$ and a projective simple representation $P\cong S_q$. Then there are two unique maximal paths starting in $q_1$ (resp. $q_2$) which go into the opposite direction to $\rho_1$ (resp. $\rho_2$). With these paths we can associate quasi-simple representations $S_1$ and $S_2$ with one-dimensional vector spaces along the support of these paths such that $\Ext(S_1,P)=\Ext(S_2,P)=\C$. It is straightforward to check that we have $P\in {}^\perp S_i$ for at least one of the two quasi-simples as the path corresponding to $S_1$ cannot end at $q_2$ if the one corresponding to $S_2$ ends in $q_1$. We can assume without loss of generality that $S_1$ satisfies the claim. We consider the representation $X_{S_1,\delta}$ having $S_1$ as a subrepresentation which means that $\dim (X_{S_1,\delta})_q=1$ for every $q\in Q_0$, $(X_{S_1,\delta})_{\rho_1}=0$ and $(X_{S_1,\delta})_{\rho}=1$ for every $\rho\neq\rho_1$.

 It is straightforward to check that the same arguments as above yield 
 \[
  P_{\Gr_{\udim P+\udim S_1}(X_{S_1,2\delta})}=2q^2+1.
 \]
\end{proof}

\begin{ex}
 If $Q=\tilde D_4$ is in subspace orientation, we can consider the following dimension vectors (and the unique indecomposables induced by them) to obtain a singular quiver Grassmannian: $\udim S=(1,1,1,0,0),\,\udim T=(1,0,0,1,1),\,\udim P=(1,0,0,1,0),\,\udim P'=(2,1,1,1,0)$.
\end{ex}

\begin{rem}
 The result $P_{\Gr_{\udim P+\udim S_1}(X_{2\delta})}=2q^2+1$ for extended Dynkin quivers with at least $3$ vertices suggests, that $\Gr_{\udim P+\udim S_1}(X_{2\delta})$ is the one point union of two rational curves. The authors have verified for extended Dynkin quivers of types $A$ and $D$ that it is indeed the one point union of two projective lines.
\end{rem}

\subsection{The Kronecker quiver}\label{kronecker}
In order to find a quiver Grassmannian with singularities for the Kronecker quiver $K(2)$, one has to consider higher dimensional representations than it is the case for other tame quivers. The smallest dimensional representation with singular quiver Grassmannian has dimension vector $3\delta=(3,3)$.
\begin{thm} \label{thm: singular quiver Grassmannians for the Kronecker quiver}
 There are quiver Grassmannians with singularities for the Kronecker quiver.
\end{thm}

\begin{proof}
 Let $X$ be the representation of $Q$ given by the following coefficient quiver $\Gamma$:
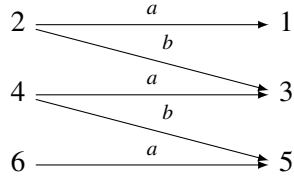
 \[
 \beginpgfgraphicnamed{tikz/fig10}
 \begin{tikzpicture}[>=latex]
  \matrix (m) [matrix of math nodes, row sep=1.3em, column sep=8em, text height=1ex, text depth=0ex]
   {      2   &   1   \\   
          4   &   3   \\   
          6   &   5   \\   
};
   \path[->,font=\scriptsize]
   (m-1-1) edge node[auto] {$a$} (m-1-2)
   (m-1-1) edge node[auto] {$b$} (m-2-2)
   (m-2-1) edge node[auto] {$a$} (m-2-2)
   (m-2-1) edge node[auto] {$b$} (m-3-2)
   (m-3-1) edge node[auto] {$a$} (m-3-2)
   ;
  \end{tikzpicture}
 \endpgfgraphicnamed
 \]
 Consider the dimension vector $\ue=(1,2)$ and the type $\ue$-subset $\beta=\{3,5,6\}$ of $\Gamma_0$. Then $C_\beta^X$ is the open dense Schubert cell of $\Gr_\ue(X)$ and every singularity of $C_\beta^X$ will be a singularity of $\Gr_\ue(X)$. As explained in section 2.3 in \cite{L13}, $C_\beta^X$ is defined by the following equations:
 \begin{align*}
  E(a,1,6): && w_{2,6} - w_{1,5} + w_{1,3} w_{4,6} \ &= \ 0 \\ 
  E(b,1,6): && w_{1,3} w_{2,6} + w_{1,5} w_{4,6}   \ &= \ 0
 \end{align*}
 Writing $x=w_{2,6}$, $y=w_{1,3}$ and $z=w_{4,6}$, we can eliminate the first equation by substituting $w_{1,5}=x+yz$ in the second equation. This identifies $C_\beta$ with the hypersurface in $\A^3$ that is defined by 
  \[
  xy+xz+yz^2 \ = \ 0.
 \]
 Its Jacobian 
 \[
  J(x,y,z) \ = \ (y+z,\ x+z^2,\ x+2yz)
 \]
 vanishes precisely in the origin $(0,0,0)$, which is a point of  the hypersurface $C_\beta^X$. Since a hypersurface in an affine space which is defined by a single equation does not have embedded components (cf.\ Exercise 5.5.I in \cite{Vakil}), $(0,0,0)$ is a singularity of $C_\beta^X$ as a variety. 
\end{proof}

\begin{cor}\label{cor: singular quiver Grassmannians in homogeneous tubes}
 Each homogeneous tube of the Kronecker quiver contains a representation with singular quiver Grassmannian.
\end{cor}

\begin{proof}
 Let $Q$ be the Kronecker quiver. Let $\delta=(1,1)$ be smallest imaginary root of $Q$. By the proof of Lemma 1.4 in \cite{Cerulli11}, the quiver Grassmannians $\Gr_\ue(X_{n\delta})$ have the same isomorphism type for fixed $\ue$ and $n$, independent from the tube that $X_{n\delta}$ lives in. Combining this with Theorem \ref{thm: singular quiver Grassmannians for the Kronecker quiver}, we see that the quiver Grassmannian $\Gr_\ue(X)$ is singular for $\ue=(1,2)$ and for every indecomposable representation $X$ of dimension ${3\delta}$. Note that every homogeneous tube contains a representation of this dimension.
\end{proof}

\begin{rem}\label{rem: families of quiver Grassmannians are not always isomorphism type preserving}
 The fact that the quiver Grassmannians for exceptional and homogeneous tubes of the Kronecker quiver are isomorphic is accidental and caused by the fact that all tubes of the Kronecker quiver have rank $1$. Indeed the argument of Lemma 1.4 in \cite{Cerulli11} shows more generally that a family of quiver Grassmannians does not deform from the homogeneous tubes to exceptional tubes of rank $1$. 
 
 This is, however, not true anymore for exceptional tubes of rank $2$. The following is an example of a family of smooth quiver Grassmannians in the homogeneous tubes that deforms to a singular quiver Grassmannian in an exceptional tube of rank $2$. Even worse, both the Poincar\'e polynomial and the Euler characteristic are not preserved by this degeneration.
\end{rem}

\begin{ex}[A family of quiver Grassmannians of type $\widetilde A_2$]\label{ex: family of quiver Grassmannians}
 Let $Q$ be a quiver of extended Dynkin quiver type $\widetilde A_2$ of the form
 \[
  \beginpgfgraphicnamed{tikz/fig6}
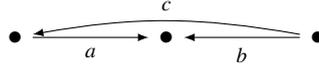
  \begin{tikzpicture}[>=latex]
   \matrix (m) [matrix of math nodes, row sep=2.5em, column sep=4em, text height=1ex, text depth=0ex]
    {      \bullet & \bullet & \bullet   \\   
    };
   \path[->,font=\scriptsize]
    (m-1-1) edge node[auto,swap] {$a$} (m-1-2)
    (m-1-3) edge node[auto] {$b$} (m-1-2)
    (m-1-3) edge[bend right=10] node[auto,swap] {$c$} (m-1-1)
    ;
   \end{tikzpicture}
  \endpgfgraphicnamed
 \]
 Let $\lambda$ be a complex parameter and $X_\lambda$ be the representation with coefficient quiver $\Gamma_\lambda$
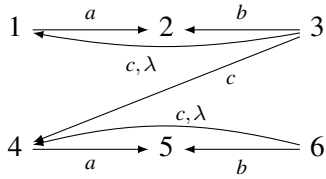
 \[
  \beginpgfgraphicnamed{tikz/fig9}
  \begin{tikzpicture}[>=latex]
   \matrix (m) [matrix of math nodes, row sep=3em, column sep=4em, text height=1ex, text depth=0ex]
    {     1 & 2 & 3   \\   
          4 & 5 & 6   \\   
    };
   \path[->,font=\scriptsize]
    (m-1-1) edge node[auto] {$a$} (m-1-2)
    (m-1-3) edge node[auto,swap] {$b$} (m-1-2)
    (m-1-3) edge[bend left=10] node[below,pos=0.6] {$c,\lambda$} (m-1-1)
    (m-1-3) edge node[right=0.3,pos=0.4] {$c$} (m-2-1)
    (m-2-3) edge[bend right=14] node[above=-0.1,pos=0.4] {$c, \lambda$} (m-2-1)
    (m-2-1) edge node[auto,swap] {$a$} (m-2-2)
    (m-2-3) edge node[auto] {$b$} (m-2-2)
    ;
   \end{tikzpicture}
  \endpgfgraphicnamed
 \]
 Then $X_\lambda$ varies through all homogeneous tubes for $\lambda\in\C^\times$and $X_0$ is in an exceptional tube. We consider the quiver Grassmannians $\Gr_\ue(X_\lambda)$ for dimension vector $\ue=(1,2,1)$. Let $\beta=\{2,4,5,6\}$. Then the Schubert cell $C_\beta^{X_\lambda}$ is open dense in $\Gr_\ue(X_\lambda)$ and we can apply the description of the quiver Grassmannian in terms of homogeneous coordinates from \cite{LW16}. 
 
 Note that we can simplify the equations of \cite{LW16} if we make use of the fact that the embedding $\Gr_\ue(X_\lambda)\to\Gr(4,6)$ factors through the product Grassmannian $\Gr(1,2)\times\Gr(2,2)\times\Gr(1,2)$, which is isomorphic to $\P^1\times\P^1$ with bihomogeneous coordinates $[\,\Delta_1:\Delta_4 \,|\, \Delta_3:\Delta_6\,]$. Then the defining bihomogeneous equation of $\Gr_\ue(X_\lambda)$ inside $\P^1\times\P^1$ is
 \[
  F_\beta(c,6,1) \ = \ \lambda\, \Delta_4\, \Delta_3 \, - \,\lambda\, \Delta_1\, \Delta_6 \, + \, \Delta_1 \, \Delta_3 \ = \ 0.
 \]
 From this, we see that $\Gr_\ue(X_\lambda)$ forms a flat family over $\C$ with respect to the parameter $\lambda$. Its fibres over $\lambda\neq 0$ are smooth quadrics, which are isomorphic to $\P^1$. The fibre over $\lambda=0$ is the transversal intersection of two projective lines in a point, which is a singularity of $\Gr_\ue(X_0)$.
 
 The Poincar\'e polynomial and the Euler characteristics of the fibres $\Gr_\ue(X_\lambda)$ in this family are:
 \begin{center}
  \begin{tabular}{c|c|c}
                   & Poincar\'e polynomial & Euler characteristic \\
       \hline 
                   &                       &        \\[-10pt]         
   $\lambda\neq 0$ & $q^2+1$               &   $2$  \\
   $\lambda=0$     & $2q^2+1$              &   $3$  
  \end{tabular}
 \end{center}
 Note that $\Gr_\ue(X_0)$ is the same quiver Grassmannian as was considered in the proof of Theorem \ref{thm: singular quiver Grassmannians for tubes of large rank}, which reproves the result for type $\widetilde A_2$.
\end{ex}

\subsection{Homogeneous tubes}
Apart from non-reduced points, it is also relatively easy to describe quiver Grassmannians coming along with representations in a homogeneous tube which are singular as a scheme. Let us consider the example from section \ref{kronecker}, i.e. the Kronecker quiver $K(2)=\xymatrix@C=1pc{0\ar[r]<0.6ex>\ar[r]<-0.1ex> & 1}$ and any representation $X_{3\delta}$ of dimension $3\delta$ which lies in a homogeneous tube. For $\ue=(1,2)$, there is a generic subrepresentation $U$ of dimension $(1,2)$ that is indecomposable projective. In particular, we have $\dim\Ext(U,X/U)=0$ because $U$ is projective. But there are also subrepresentations of dimension $\ue$ which are isomorphic to the direct sum $X_\delta\oplus S_1$ where $X_\delta$ is the indecomposable representation of dimension $\delta$ lying in the same tube and where $S_1$ is the simple projective representation supported at the sink $1$ of $K(2)$. As $X_\delta\oplus S_1$ is also a subrepresentation of $X_{2\delta}$ with the simple injective quotient $S_0$, there exists a commutative diagram
\[\xymatrix{0\ar[r]&S_0\ar[r]& S_0\oplus X_\delta\ar[r] &X_\delta\ar[r]& 0\\0\ar[r]&X_{2\delta}\ar@{^-->>}[u]\ar[r]&X_{3\delta}\ar@{^-->>}[u]\ar[r]&X_\delta\ar@{=}[u]\ar[r]& 0\\0\ar[r]&X_\delta\oplus S_1\ar@{^(->}[u]\ar@{=}[r]&X_\delta\oplus S_1\ar@{^(->}[u]\ar[r]&0\ar[u]\ar[r]& 0}\] 
Since we have $\dim\Ext(X_\delta\oplus S_1,S_0\oplus X_\delta)=1$, this shows that the quiver Grassmannian $\Gr_{(1,2)}(X_{3\delta})$ is singular as a scheme, see \cite[Proposition 6]{cr}. Since it is not clear that the scheme is reduced, this observation does not imply that $\Gr_{(1,2)}(X_{3\delta})$ is singular as a variety. But as we observed in section \ref{kronecker}, it is also smooth as a variety.

In order to construct a singular quiver Grassmannian for homogeneous tubes, we make use of the following lemma.
\begin{lemma}\label{schofield}  Let $X$ and $Y$ be two exceptional representations of a quiver $Q$ such that $X\in {}^\perp Y$, $\dim\Ext(Y,X)=m$ and $\mathrm{supp} (X)\cap\supp (Y)=\emptyset$. Let $\ue=a\cdot\udim X+b\cdot\udim Y$. Then there is a fully faithful functor $F:\Rep(K(m))\to\Rep(Q)$ inducing isomorphisms $\Gr_\ue(FZ)\cong \Gr_{(b,a)}(Z)$ for every representation $Z\in\Rep(K(m))$.
\end{lemma}
\begin{proof}The existence of the functor $F$ is ensured by Schofield induction \cite{sc2}.
A fixed representation $Z\in\Rep(K(m))$ of dimension $(r,s)$ gives rise to a short exact sequence
\[\ses{X^s}{FZ}{Y^r}\]
and induces a map $\Psi_\ue:\Gr_\ue(FZ)\to\coprod_{\underline f+\underline g=\ue}\Gr_{\underline f}(X^s)\times\Gr_{\underline g}(Y^r)$. Let $\ue=a\cdot\udim X+b\cdot\udim Y$. Since $F$ is fully faithful, every subrepresentation $U$ of $Z$ of dimension $(b,a)$ corresponds to a subrepresentation $FU$ of $FZ$ and we get an embedding $\Gr_{(b,a)}(Z)\hookrightarrow\Gr_{\ue}(FZ)$. Indeed, every subrepresentation $U$ of $X$ of dimension $(b,a)$ gives rise to a commutative diagram
\[\xymatrix{0\ar[r]&X^s\ar[r]& FZ\ar[r] &Y^r\ar[r]& 0\\0\ar[r]&X^a\ar@{^(->}[u]\ar[r]&FU\ar@{^(->}[u]\ar[r]&Y^b\ar@{^(->}[u]\ar[r]& 0}\] 
Since we have $\supp (X)\cap\supp (Y)=\emptyset$, the equality $\underline f+\underline g=\ue$ is only satisfied if $\underline f=a\cdot\udim X$ and $\underline g=b\cdot\udim Y$. Since every subrepresentation of dimension $a\cdot\udim X$ of $X^s$ is isomorphic to $X^a$ as $X$ is exceptional and since the analogous statement is true for subrepresentations of dimension $b\cdot\udim Y$ of $Y^r$, it follows that every subrepresentation is of this shape. Finally, we have 
\[\Psi_\ue^{-1}(X^a,Y^b)=\mathbb A^{[Y^b,X^{s-a}]}=\mathbb A^0\] which yields the claim.
\end{proof}
We make use of this lemma to prove the following
\begin{thm} \label{thm: singular quiver Grassmannians in homogeneous tubes}
 Let $Q$ be a quiver of extended Dynkin type . Then there exists a quiver Grassmannian $\Gr_{\ue}(X)$ which is singular as a variety and where $X$ is an indecomposable representation lying in a homogeneous tube.
\end{thm}
\begin{proof}If $Q=K(2)$, then this is Corollary \ref{cor: singular quiver Grassmannians in homogeneous tubes}. Thus assume $|Q|\geq 3$ and let $\delta$ be the unique imaginary Schur root of $Q$. Then there exists at least one source or sink $q\in Q_0$ with $\delta_q=1$, see for instance \cite[Section 4]{Crawley-Boevey92} for a list of the imaginary Schur roots of extended Dynkin quivers. Denote by $e_q$ the corresponding simple root and by $S_q$ the simple representation corresponding to $q$. It is straightforward to check that $\alpha:=\delta-e_q$ is a root of the corresponding Dynkin quiver which is clearly exceptional as a root of a Dynkin quiver. Let $X_\alpha$ be the exceptional representation of dimension $\alpha$. Then we have $[X_\alpha,S_q]=[S_q,X_\alpha]=0$ because $\mathrm{supp}(X_\alpha)\cap\mathrm{supp}(S_q)=\emptyset$. Depending on the orientation of $Q$ and as we have $\sum_{q'\in N_q}\delta_{q'}=2$, it follows that $\dim\Ext(X_\alpha,S_q)=2$ or $\dim\Ext(S_q,X_\alpha)=2$. Thus the functor $F:\Rep(K(2))\to\Rep(Q)$, restricted to representations of dimension $(3,3)$, induces a $\P^1$-family of non-isomorphic indecomposable representations of dimension $3\delta$. In particular, the image of $F$ contains representations of dimension $3\delta$ which lie in a homogeneous tube. By Lemma \ref{schofield}, we have 
$$\Gr_{\alpha+2e_q}(X_{3\delta})\cong \Gr_{(1,2)}(X_{(3,3)}).$$
As this quiver Grassmannian is singular by Corollary \ref{cor: singular quiver Grassmannians in homogeneous tubes}, the result follows.
\end{proof}


\section{Negative Euler characteristics for wild quivers}


\subsection{Generalized Kronecker quiver}

The case of wild Kronecker quivers is based on the following theorem by Hille.

\begin{thm}[{\cite[Thm.\ 1.2]{Hille15}}]\label{thm: Hille}
 Let $n\geq1$ and $Q$ the $n$-Kronecker quiver. Then every projective subscheme of $\P^{n-1}$ is isomorphic to the quiver Grassmannian $\Gr_\ue(X)$ for some representation $X$ and some dimension vector $\ue$ of $Q$.
\end{thm}

\begin{cor}\label{negativeKronecker}
 Let $n\geq3$ and $Q$ be the $n$-Kronecker quiver. Then every integer can be realized as the Euler characteristic of a quiver Grassmannian of $Q$.
\end{cor}

\begin{proof}
 Since every closed subscheme of $\P^2$ can be realized as a closed subscheme of $\P^{n-1}$ for $n\geq3$, it suffices to prove the theorem for $n=3$.
 
 It is well-known that there are curves of arbitrarily negative Euler characteristic in $\P^2$. Let $k$ be an integer and $X$ a curve with Euler characteristic $\chi(X)\leq k$. Define $Y$ as the disjoint union of $X$ with $k-\chi(X)$ points in $\P^n$. By Theorem \ref{thm: Hille}, $Y\simeq\Gr_\ue(X)$ for some representation $X$ and some dimension vector $\ue$ of $Q$, and thus $\chi\bigl(\Gr_\ue(X)\bigr)=\chi(Y)=k$ as desired.
\end{proof}

\subsection{Minimal wild quivers}
In this section, we show that, for every minimal wild quiver, there exists an indecomposable representation $X$ and a dimension vector $e$ such that $\chi(\Gr_e(X))<0$. The idea is to combine Schofield induction and the Caldero-Chapoton map for quiver Grassmannians. The following fact is easily deduced from the well-understood representation theory of extended Dynkin quivers:
\begin{lemma}\label{roots}

Let $\alpha$ be a preprojective root of an extended Dynkin quiver and let $\delta$  be the unique imaginary Schur root. 
\begin{enumerate}
\item Then there exists an $n\in\N$ such that $n\delta <\alpha <(n+1)\delta$.
\item The dimension vector $(n+1)\delta-\alpha$ is a preinjective root.
\item The dimension vector $\alpha+n\delta$ is a preprojective root for all $n\in\N$.
\end{enumerate}
\end{lemma}

\begin{prop} \label{prop: roots with large extension for minimal wild quiver}
 For every minimal wild quiver with at least $3$ vertices and every $m\geq 1$, there exist two exceptional roots $\alpha$ and $\beta$ such that $\mathrm{supp}(\alpha)\cap\mathrm{supp}(\beta)=\emptyset$, $\alpha\in\beta^\perp$,  $\hom(\alpha,\beta)=0$ and $\mathrm{ext}(\alpha,\beta)\geq m$.
\end{prop}

\begin{proof}
 If $\hat Q$ is minimal wild with at least $3$ vertices, there exists an extended Dynkin quiver $Q$ such that $\hat Q$ is obtained by either adding an arrow between an existing vertex and a new vertex or by adding an arrow between two existing vertices. 

 In the first case, we can decompose $\hat Q_0= Q_0\cup \{q\}$ where $ Q$ is an extended Dynkin quiver. Moreover, $q$ is connected to a vertex $q'\in Q_0$ by at least one arrow. By Lemma \ref{roots}, there exists a preprojective (and thus exceptional) root $\alpha$ of $Q$ such that $\alpha_q\geq m$. If $s_q$ is the simple root corresponding to $q$, depending on the orientation of the connecting arrows, we either have $\mathrm{ext}(\alpha,s_q)\geq m$ or $\mathrm{ext}(s_q,\alpha)\geq m$. 

 In the second case and if the new arrow is between two vertices which were already connected by an arrow, the quiver $\hat Q$ is forced to have a subquiver of one of the following forms:
 \[
  \bullet \Longrightarrow\bullet\longleftarrow\bullet,\quad \bullet \Longrightarrow\bullet\longrightarrow\bullet,\quad \bullet \longrightarrow\bullet\Longrightarrow\bullet,\quad \text{or} \quad \bullet \longleftarrow\bullet\Longrightarrow\bullet .
 \]
 Thus it has a Kronecker quiver as a subquiver which means that we can apply the argument from above. If the vertices were not connected before, the quiver is forced to have an undirected cycle as the quiver itself was connected before. As $Q$ is of extended Dynkin type, $\hat Q$ cannot be of type $\tilde A_n$. Thus the new quiver has a proper subquiver of type $\tilde A_n$ for some $n\geq 2$ which is connected to an additional vertex. Thus we can apply the argument from above. 
\end{proof}

Combining the results of this section with Corollary \ref{negativeKronecker} we obtain:

\begin{thm} \label{thm: negative Euler characteristics for wild quivers}
 For every wild quiver and every $k\in\Z$, there exists a quiver Grassmannian with Euler characteristic $k$. In particular, there are quiver Grassmannians of $Q$ that do not have a cell decomposition into affine spaces.
\end{thm}

\begin{cor}\label{cor: singular quiver Grassmannian for wild quivers}
 Every wild quiver has quiver Grassmannians with singularities.
\end{cor}

\begin{proof}
 Since there exist singular closed curves in $\P^2$, this follows immediately from Theorems \ref{negativeKronecker} and \ref{thm: negative Euler characteristics for wild quivers}.
\end{proof}

\begin{rem}
 The proof of Proposition \ref{prop: roots with large extension for minimal wild quiver} implies actually the stronger statement that for any wild acyclic quiver $Q$ with at least $3$ vertices, any $m\geq1$, any representation $X$ of the $m$-Kronecker quiver and any dimension vector $\ue$ for the Kronecker quiver, the quiver Grassmannian $\Gr_\ue(X)$ is isomorphic to a quiver Grassmannian for $Q$. Combining this with Theorem \ref{thm: Hille} shows that every projective scheme is isomorphic to a quiver Grassmannian for $Q$.
 
 This result is the theme of the recent paper \cite{Ringel17} of Ringel that was proven independently from the present work. One of the main ideas of Ringel is comparable to the one of Lemma \ref{schofield}. Under certain additional assumptions on $X$ and $Y$, he extends Lemma \ref{schofield} to generalized Kronecker quivers in special case $(b,a)=(1,1)$.
\end{rem}

\begin{small}

\bibliographystyle{plain}
\def\cprime{$'$}

\end{small}

\end{document}